\documentclass[reqno]{amsart}
\usepackage{amsmath, amsthm, amsfonts, amssymb}
\usepackage{enumerate, mathrsfs, array}
\usepackage{xcolor}
\newcommand{\mC}{\ensuremath{\mathbb{C}}}
\newcommand{\mD}{\ensuremath{\mathbb{D}}}

\newcommand{\mN}{\ensuremath{\mathbb{N}}}

\newcommand{\cM}{\ensuremath{\mathcal{M}}}
\newcommand{\cF}{\ensuremath{\mathcal{F}}}

\newcommand{\cH}{\ensuremath{\mathcal{H}}}

\newtheorem{thm}{\bf Theorem}

\newtheorem{theorem}{Theorem}[section]
\newtheorem{lemma}[theorem]{Lemma}
\newtheorem{question}[theorem]{Question}

\newtheorem{corollary}[theorem]{Corollary}
\theoremstyle{definition}
\newtheorem{definition}[theorem]{Definition}
\newtheorem{example}[theorem]{Example}

\theoremstyle{remark}

\numberwithin{equation}{section}
 \allowdisplaybreaks
 
\begin{document}

\title[A Note on Normal Families...]{A Note on Normal Families Concerning Nonexceptional Functions and Differential Polynomials}

\author[N. Bharti]{Nikhil Bharti}
\address{
\begin{tabular}{lll}
&Nikhil Bharti\\
&Department of Mathematics\\
&University of Jammu\\
&Jammu-180 006, INDIA\\
&{\em Orcid id}: 0000-0003-2501-6247\\ 
\end{tabular}}
\email{nikhilbharti94@gmail.com}

\author[A. Singh]{Anil Singh}
\address{
\begin{tabular}{lll}
&Anil Singh\\
&Department of Mathematics\\
&Maulana Azad Memorial College\\
&Jammu-180 006, INDIA\\ 
&{\em Orcid id}: 0000-0002-3447-8585
\end{tabular}}
\email{anilmanhasfeb90@gmail.com}

\begin{abstract}
We study normality of a family of  meromorphic functions, whose differential polynomials satisfy a certain condition, which significantly improves and generalizes some recent results of Chen (Filomat, 31(14) 2017, 4665-4671).  Moreover, we demonstrate, with the help of examples, that the result is sharp.
\end{abstract}

\renewcommand{\thefootnote}{\fnsymbol{footnote}}
\footnotetext{2010 {\it Mathematics Subject Classification}. Primary 30D45; Secondary 30D30, 30D35, 34M05.}
\footnotetext{{\it Keywords and phrases}. Normal families, differential polynomials, meromorphic functions.}

\maketitle

\section{Introduction and main results}
A family $\mathcal{F}$ of meromorphic functions in a domain $D\subseteq\mC$ is said to be {\it normal} in $D$ if from each sequence of functions in $\mathcal{F},$ one can extract a subsequence which converges locally uniformly in $D$ with respect to the spherical metric to the limit function which is either meromorphic in $D$ or the constant $\infty.$ A family $\cF$ is normal at a point $z_0\in D$ if and only if it is normal in some neighbourhood of $z_0$ (see \cite{schiff, zalcman}). The theory of normal families plays an important role in complex analysis and has been used extensively to study the Riemann mapping theorem, dynamical systems, conformal mappings and many other interesting topics.

We denote by  $\mathcal{H}(D)$ (respectively, $\mathcal{M}(D)$), the class of all holomorphic (respectively, meromorphic) functions on  domain $D\subseteq\mC,$ and $\mathbb{D}$ shall denote the open unit disk in $\mathbb{C}.$

 Let $f,~g\in\cM(D).$ Then $g$ is said to be an exceptional function of $f$ in $D$ if $f(z)-g(z)\neq 0$ in $D,$ otherwise, we call $g$ a non-exceptional function of $f$ in $D.$ In particular, if $g\equiv c,$ where $c\in\mC,$ then $c$ is said to be an exceptional (respectively, non-exceptional) value of $f$ in $D.$ Our objective in this paper is to study the normality of a family of non-vanishing meromorphic functions in a domain $D\subseteq\mC$ whose differential polynomials have a non-exceptional meromorphic function in $D.$ 
\medskip

\begin{definition}\cite{grahl}
Let $k\in\mN$ and $n_0, n_1,\ldots, n_k$ be non-negative integers,  not all zero. Let $f\in\mathcal{M}(D).$ Then the product $$M[f]:=a\cdot\prod\limits_{j=0}^{k}\left(f^{(j)}\right)^{n_j}$$ is called a differential monomial of $f,$ where $a~(\not\equiv 0)\in\mathcal{M}(D).$ If $a\equiv 1,$ then $M[f]$ is said to be a normalized differential monomial of $f.$ The quantities $$d:=\sum\limits_{j=0}^{k}n_j \mbox{ and } w:=\sum\limits_{j=0}^{k}(j+1)n_j$$ are called the degree and weight of the differential monomial $M[f],$ respectively. 

For $1\leq i\leq m,$ let $M_i[f]$  be $m$ normalized differential monomials of $f$ having degree $d_i$ and weight $w_i.$ Then the sum $$P[f]:= \sum\limits_{i=1}^{m}a_iM_i[f]$$ is called a differential polynomial of $f$ and the quantities $$d:=\max\left\{d_i: 1\leq i\leq m\right\};~ d_0:=\min\left\{d_i: 1\leq i\leq m\right\} \mbox{ and } w:=\max\left\{w_i: 1\leq i\leq m\right\}$$ are called the degree, the lower degree and the weight of the differential polynomial $P[f],$ respectively.  If $d_1=d_2=\cdots=d_m,$ then $P[f]$ is said to be a homogeneous differential polynomial.
\end{definition}

\medskip

In the present paper, we consider the differential polynomials of the form 

\begin{equation}\label{eqn:3}
Q[f]:=f^{p_0}(f^{p_1})^{(q_1)}(f^{p_2})^{(q_2)}\cdots(f^{p_k})^{(q_k)},
\end{equation}
where $p_0,~p_1,~\ldots,~p_k,~q_1,~q_2,~\ldots,~q_k$ are non-negative integers. 

Set $p'=\sum\limits_{i=1}^{k} p_i$ and  $q'=\sum\limits_{i=1}^{k}q_i.$ Further, we assume that $p_i\geq q_i$ for $i=1,~2,~\ldots,~k$ and $q'>0.$ Through Leibniz rule, one can see that $$(f^{p_i})^{(q_i)}=\sum\limits_{n_1+n_2+\cdots+n_{p_i}=q_i} \frac{q_i!}{n_1!n_2!\cdots n_{p_i}!}f^{(n_1)}f^{(n_2)}\cdots f^{(n_{p_i})},$$ where $n_i$'s are non-negative integers. Thus the degree of $Q[f],$ $d=p_0+p'$ and the weight of $Q[f]$, $w=p_0+p'+q'=d+q'.$\\
In literature, the differential polynomial $Q[f]$ was first considered by Dethloff et al in \cite{dethloff} and since then, it has grasped interest of many authors, see \cite{charak, thin-1, thin-2, thin-3}. However, in all the previous works, there is an additional condition on $Q[f],$ namely that $p_0\neq 0.$ We emphasize that this condition is not required in our results. 

\medskip

In \cite{deng}, Deng et al. obtained the following normality criterion:

\begin{thm}\cite[Theorem 5]{deng}\label{thm:deng}
Let $\mathcal{F}\subset\mathcal{M}(D)$ be a family of non-vanishing functions, let $h\in\mathcal{H}(D)$ be such that $h\not\equiv 0,$ and let $k$ be a positive integer. If, for every $f\in\mathcal{F},$ $f^{(k)}-h$ has at most $k$ distinct zeros, ignoring multiplicities in $D,$ then $\mathcal{F}$ is normal in $D.$
\end{thm}

The proof of Theorem \ref{thm:deng} when $h\equiv 1$ is due to Chang \cite{chang}. In fact, Theorem \ref{thm:deng} generalized a normality criterion of Gu \cite{gu} who obtained Theorem \ref{thm:deng} under the hypothesis that $f^{(k)}-1$ has no zeros in $D,$ and thus gave an affirmative answer to a question by Hayman (see \cite[Problem 5.11]{hayman-3}). Thin \cite{thin-3} extended Theorem \ref{thm:deng} to differential polynomials when $h\equiv 1.$ Recently, Chen \cite{chen} considered a non-exceptional meromorphic function instead of a non-exceptional holomorphic function in Theorem \ref{thm:deng} and obtained the following normality criterion:

\begin{thm}\cite[Theorem 2]{chen}\label{thm:chen}
Let $\mathcal{F}\subset\mathcal{M}(D)$ be a family of non-vanishing functions, all of whose poles are multiple, and let $h\in\mathcal{M}(D)$ be such that $h\not\equiv 0,~\infty$ and all poles of $h$ are simple. Let $k$ be a positive integer. If, for every $f\in\mathcal{F},$ $f^{(k)}-h$ has at most $k$ distinct zeros, ignoring multiplicities, in $D,$ then $\mathcal{F}$ is normal in $D.$
\end{thm}

\medskip

Since a differential polynomial is an extension of a derivative, a natural question about Theorem \ref{thm:chen} arises:

\begin{question}\label{q:1}
Does Theorem \ref{thm:chen} remain valid if the $k$-th derivative $f^{(k)}$ is replaced by the differential polynomial $Q[f]?$ 
\end{question}

We intend to give an affirmative answer to Question \ref{q:1}. For this, we first establish the following theorem:

\begin{theorem}\label{thm:7}
Let $\cF\subset\mathcal{M}(D)$ be a family of non-vanishing functions, all of whose poles are multiple and let $h\in\mathcal{M}(D)$ be such that $h(z)\neq 0$ in $D,$ $h\not\equiv\infty,$ and poles of $h$ are simple. If, for each $f\in\cF,$ $Q[f]-h$ has at most $w-1$ zeros, ignoring multiplicities, in $D,$ then $\cF$ is normal in $D.$ 
\end{theorem}

In what follows, we show that various conditions in the hypothesis of Theorem \ref{thm:7} are essential.

\begin{example}
Consider the family $$\mathcal{F}:=\left\{f_j: f_j(z)=\frac{1}{jz},~ z\in\mD,~j\geq 3,~j\in\mN\right\}$$ and let $Q[f_j](z):=\left(f_j^2\right)''=2f_jf_j''+2\left(f_j'\right)^{2}$ so that $w=4.$ Then $\mathcal{F}\subset\mathcal{M}(\mD)$ and each $f_j$ has only simple poles. Also, $Q[f_j](z)=6/j^2z^4.$ Take $h(z)=1/z.$ Then $Q[f_j](z)-h(z)=(6-j^2z^3)/j^2z^4$ has exactly $w-1=3$ distinct zeros in $\mD.$ However, the family $\mathcal{F}$ is not normal in $\mD,$ {\color[rgb]{0,0,1}since $f_j(0)=\infty$ and $f_j(z)\rightarrow 0$ in $\mD\setminus\{0\}.$} This shows that the condition that poles of each $f\in\cF$ are multiple, is necessary. 
\end{example}

\begin{example}
Let $\mathcal{F}:=\left\{f_j: f_j(z)=1/jz^2,~ z\in\mD,~j\geq 2,~j\in\mN\right\}$ and let $Q[f_j](z):=f_jf_j'.$ Then $w=3$ and $Q[f_j](z)=-2/j^2z^5.$ Take $h(z)=1/z^3.$ Then $Q[f_j](z)-h(z)=(-2-j^2z^2)/j^2z^5$ has exactly $w-1=2$ zeros in $\mD,$ ignoring multiplicities. But the family $\mathcal{F}$ is not normal in $\mD,$ {\color[rgb]{0,0,1}since $f_j(0)=\infty$ and $f_j(z)\rightarrow 0$ in $\mD\setminus\{0\}.$} This shows that the condition that poles of $h$ are simple cannot be avoided. 
\end{example}

The following example shows that Theorem \ref{thm:7} fails to hold if functions in $\mathcal{F}$ are permitted to have zeros.

\begin{example}
Let $\mathcal{F}:=\left\{f_j: f_j(z)=j(\tan 2z)^2,~ z\in\mD,~j\in\mN\right\}.$ Then for each $j,$ $f_j$ has multiple zeros and multiple poles in $\mD.$ Let $Q[f_j](z):=f_jf_j'.$ Then $w=3$ and $Q[f_j](z)=4j^2(\tan 2z)^3 \sec^2 2z.$ Take $h(z)=-4/\sin 2z \cos 2z.$ Then $Q[f_j](z)-h(z)=((4j^2(\sin 2z)^4)+4(\cos 2z)^4)/\sin 2z (\cos 2z)^5$ has at most $w-1=2$ distinct zeros in $\mD.$ However, $\mathcal{F}$ is not normal in $\mD,$ {\color[rgb]{0,0,1}since $f_j(0)=0$ and $f_j(z)\rightarrow\infty$ in $\mD\setminus\{0\}.$}
\end{example}

The next example demonstrates that the condition ``$Q[f]-h$ has at most $w-1$ distinct zeros in D" in Theorem \ref{thm:7} is best possible:

\begin{example}\label{exp:d3}
Consider the family $$\mathcal{F}:=\left\{f_j: f_j(z)=\frac{1}{jz^2},~ z\in\mD,~j\geq 3,~j\in\mN\right\}$$ and let $Q[f_j](z):=f_j'.$ Then $w=2$ and $Q[f_j](z)=-2/jz^3.$ Let $h(z)=1/z.$ Then $Q[f_j](z)-h(z)=(-2-jz^2)/jz^3$ has exactly $w=2$ distinct zeros in $\mD.$ However, the family $\mathcal{F}$ is not normal in $\mD,$ {\color[rgb]{0,0,1}since $f_j(0)=\infty$ and $f_j(z)\rightarrow 0$ in $\mD\setminus\{0\}.$}
\end{example}

{\color[rgb]{0,0,1}In \cite{bharti}, the first author recently proved the following result:

\begin{thm}\label{thm:bharti}\cite[Theorem 1]{bharti}
Let $\left\{f_j\right\}\subset\mathcal{M}(D)$ and $\left\{h_j\right\}\subset\mathcal{H}(D)$ be such that $h_j\rightarrow h$ locally uniformly in $D,$ where $h\in\mathcal{H}(D)$ and $h\not\equiv 0.$ Let $Q[f_j]$ be a differential polynomial of $f_j$ as defined in \eqref{eqn:3} having degree $d>1$ and weight $w.$ If, for each $j,$ $f_j(z)\neq 0$ and $Q[f_j]-h_j$ has at most $w-1$ zeros, ignoring multiplicities, in $D,$ then $\left\{f_j\right\}$ is normal in $D.$
\end{thm}

In particular, if $\mathcal{F}\subset\mathcal{M}(D)$ is a family of non-vanishing functions and if $h\in\cH(D),~h\not\equiv 0$ such that, for each $f\in\mathcal{F},$ $Q[f]-h$ has at most $w-1$ zeros, ignoring multiplicities, in $D,$ then the normality of $\mathcal{F}$ in $D$ follows immediately from Theorem \ref{thm:bharti}. To see this, suppose, on the contrary, that $\mathcal{F}$ is not normal at $z_0\in D.$ Then by Zalcman's lemma \cite[p. 101]{schiff}, there exist a sequence of points $\left\{z_j\right\}\subset D$ with $z_j\rightarrow z_0,$ a sequence of positive real numbers $\rho_j\rightarrow 0,$ and a sequence $\left\{f_j\right\}\subset\mathcal{F}$ such that the sequence $g_j(\zeta):=f_j(z_j+\rho_j\zeta)$ converges spherically locally uniformly in $\mathbb{C}$ to a non-constant meromorphic function $g(\zeta).$ On the other hand, by assumption, we have $Q[f_j]-h$ has at most $w-1$ distinct zeros in $D$ and so by Theorem \ref{thm:bharti}, $\{f_j\}$ is normal in $D.$ Then, by the converse of Zalcman's lemma \cite[p. 103]{schiff}, the sequence $f_j(z_j+\rho_j\zeta)$ must converge spherically locally uniformly in $\mathbb{C}$ to a constant, which is a contradiction to the fact that $g$ is non-constant.} 


\medskip

It is worthwhile to note that if functions in $\cF$ have only multiple poles, then the degree of $Q[f]$ in Theorem \ref{thm:bharti} can be taken to be one. Combining Theorems \ref{thm:7} and \ref{thm:bharti}, and using the fact that normality is a local property, we get the following result which gives an affirmative answer to Question \ref{q:1}.

\begin{theorem}\label{thm:nik-anil}
Let $\cF\subset\mathcal{M}(D)$ be a family of non-vanishing functions, all of whose poles are multiple and let $h\in\mathcal{M}(D)$ be such that all poles of $h$ are simple and $h\not\equiv 0,~\infty.$ If, for each $f\in\cF,$ $Q[f]-h$ has at most $w-1$ zeros, ignoring multiplicities, in $D,$ then $\cF$ is normal in $D.$ 
\end{theorem}
 
Taking $Q[f]=f^{(k)}$ in Theorem \ref{thm:nik-anil}, we recover Theorems \ref{thm:deng} and \ref{thm:chen} depending upon whether $h\in\cH(D)~\mbox{or }\cM(D).$ Moreover, from Theorem \ref{thm:nik-anil}, we immediately get the following:

\begin{corollary}
Let $\mathcal{F}\subset\mathcal{M}(D)$ be a family of non-vanishing functions, all of whose poles are multiple. If, for every $f\in\mathcal{F},$ $Q[f]$ has at most $w-1$ fixed points in $D,$ then $\mathcal{F}$ is normal in $D.$
\end{corollary}

\section{Preliminary lemmas}
 To prove the main results, we need some preliminary lemmas. Here, we assume that the standard notations of the Nevanlinna's value distribution theory like $m(r,f),~N(r,f),~T(r,f),~S(r,f)$ (see \cite{hayman}) are known to the reader. Recall that if $f,~\phi\in\cM(\mC),$ then we say that $\phi$ is a small function of $f$ if $T(r,\phi)=S(r,f),$ as $r\rightarrow\infty$ possibly outside a set of finite Lebesgue measure. Furthermore, the quantities $$\delta(a;f):=1-\limsup\limits_{r\to\infty}\frac{N\left(r,1/(f-a)\right)}{T(r,f)}$$ and $$\Theta(a;f):=1-\limsup\limits_{r\to\infty}\frac{\overline{N}\left(r,1/(f-a)\right)}{T(r,f)},$$ where $a\in\overline{\mC},$ are called the deficiency and the truncated defect of $f$ at $a,$ respectively.

The following lemma is an extension of the Zalcman-Pang Lemma due to Chen and Gu \cite{chen-gu} (cf. \cite[Lemma 2]{pang-zalcman}).

\begin{lemma}[Zalcman-Pang Lemma]\label{lem:zp}
Let $\mathcal{F}\subset\mathcal{M}(D)$ be such that each $f\in\cF,$ has zeros of multiplicity at least $m$ and poles of multiplicity at least $p.$ Let $-p<\alpha<m.$ If $\mathcal{F}$ is not normal at $z_0\in D,$ then there exist
sequences $\left\{f_j\right\}\subset\mathcal{F},$ $\left\{z_j\right\}\subset D$ satisfying $z_j\rightarrow z_0$ and positive numbers $\rho_j$ with $\rho_j\rightarrow 0$ such that the sequence $\left\{g_j\right\}$ defined by $$g_j(\zeta):=\rho_{j}^{-\alpha}f_j(z_j+\rho_j\zeta)\rightarrow g(\zeta)$$ locally uniformly in $\mathbb{C}$ with respect to the spherical metric, where $g$ is a non-constant meromorphic function on $\mathbb{C}$ such that for every $\zeta\in\mathbb{C},$ $g^{\#}(\zeta)\leq g^{\#}(0)=1.$

If functions in $\cF$ have no zero, then $\alpha\in (-p,+\infty).$
 \end{lemma}

\begin{lemma}\cite[Theorem 1.1]{lahiri}\label{lem:anil}
  Let $f\in\cM(\mC)$ be a transcendental function and $\phi~(\not\equiv 0)$ be a small function of $f$ such that $f$ and $\phi$ do not have any common pole. Let $P[f]$ be a differential polynomial of $f$ and $\delta(0; f) d_0+ \Theta(\infty; f) > 1$, where $d_0 (\geq 1)$ is the lower degree of $P[f]$. Then $P[f] - \phi$ has infinitely many zeros.
 \end{lemma}
 
It is noteworthy to mention that after going through the proof of Lemma \ref{lem:anil}, we note that the condition ``$f$ and $\phi$ do not have any common pole" is not required.

\begin{lemma}\label{lem:nik}
Let $f\in\mathcal{M}(\mC)$ be a non-vanishing transcendental function, all of whose poles are multiple and let $\phi~(\not\equiv 0)$ be a small function of $f.$ Then $Q[f]-\phi$ has infinitely many zeros in $\mC.$
\end{lemma}
 
\begin{proof}
In view of Lemma \ref{lem:anil}, we only need to show that $\delta(0; f)d + \Theta(\infty; f) > 1,$ where $d$ is the degree of $Q[f].$
Since $f$ is non-vanishing, by definition of $\delta(0;f),$ it follows that $\delta(0;f)=1.$ Further, since all poles of $f$ are multiple, one can easily see that $\Theta(\infty; f)\geq 1/2$ and hence $\delta(0; f)d + \Theta(\infty; f)\geq 3/2 >1,$ as desired.
\end{proof}   
		
Using the method of Chang \cite{chang} with significant modifications, we obtain the following lemma which generalizes Lemma 2.4 in \cite{chang} and Lemma 11 in \cite{deng}.

\begin{lemma}\label{lem:2}
	Let
\begin{equation}\label{eq:l1.1}
R(z)=\frac{A\prod\limits_{i=1}^{s}\left(z+\alpha_i\right)^{n_i}}{\prod\limits_{i=1}^{t}\left(z+\beta_i\right)^{m_i}}
\end{equation}
be a rational function and let $\sum\limits_{i=1}^{s}n_i=N,$ $\sum\limits_{i=1}^{t}m_i=M$ and assume that $M-N\geq t.$ Let $h~(\not\equiv 0)$ be a polynomial. Let $Q[R]$ be a differential polynomial of $R$ as defined in \eqref{eqn:3} having degree $d$ and weight $w.$ Then $Q[R]-h$ has at least $w$ distinct zeros in $\mathbb{C}.$  
\end{lemma}

\begin{proof}
Let 
	\begin{equation}\label{eq:l1.2}
		h(z)=B\prod\limits_{i=1}^{v}(z+ \gamma_i)^{v_i},
	\end{equation}
	where $B$ is a non-zero constant and $\gamma_i$ ($1\leq i\leq v$) are distinct complex numbers.
	
	From \eqref{eq:l1.1}, we deduce that 
	\begin{equation}\label{eq:l1.3}
		Q[R](z)=\frac{A\prod\limits_{i=1}^{s}\left(z+\alpha_i\right)^{(n_i+1)d-w}}{\prod\limits_{i=1}^{t}\left(z+\beta_i\right)^{(m_i-1)d+w}}\cdot G_R(z),
	\end{equation}
	where  $G_R$ is a polynomial of degree at most $(s+t-1)(w-d).$
	
	Also, it is easy to see that $Q[R]-h$ has at least one zero in $\mathbb{C}.$ Thus we may set
	\begin{equation}\label{eq:l1.4}
		Q[R](z)=h(z)+\frac{C\prod\limits_{i=1}^{q}(z+\delta_i)^{q_i}}{\prod_{i=1}^{t}(z+\beta_i)^{(m_i-1)d+w}},
	\end{equation}
	where $C\in\mathbb{C}\setminus\left\{0\right\},$ $q_i$ are positive integers and $\delta_i$ ($1\leq i\leq q$) are distinct complex numbers.
	
	Now, from \eqref{eq:l1.2}, \eqref{eq:l1.3} and \eqref{eq:l1.4}, we get
	\begin{equation}\label{eq:l1.5}
		B\prod\limits_{i=1}^{v}(z+\gamma_i)^{v_i}\prod_{i=1}^{t}(z+\beta_i)^{(m_i-1)d+w} - A\prod\limits_{i=1}^{s}\left(z+\alpha_i\right)^{(n_i+1)d-w}G_R(z)=-C\prod\limits_{i=1}^{q}(z+\delta_i)^{q_i}
	\end{equation}
 Let $V$ be the degree of $h.$ Then, since $\sum\limits_{i=1}^{t}m_i \geq \sum\limits_{i=1}^{s}n_i,$ we have
\begin{align}\label{eq:l1.a}
V &> (s-1)(w-d)-s(w-d)\nonumber\\
&= (s+t-1)(w-d)-t(w-d)-s(w-d)\nonumber\\
\Rightarrow V+t(w-d)&> (s+t-1)(w-d)-s(w-d)\nonumber\\
\Rightarrow V+t(w-d)+\sum\limits_{i=1}^{t}m_i &> \sum\limits_{i=1}^{s}n_i+(s+t-1)(w-d)-s(w-d)\nonumber\\
\Rightarrow \sum\limits_{i=1}^{v}v_i+\sum\limits_{i=1}^{t}[(m_i-1)d+w] &>
 \sum\limits_{i=1}^{s}[(n_i+1)d-w] + (s+t-1)(w-d).
\end{align}
From \eqref{eq:l1.5} and \eqref{eq:l1.a}, we get $$\sum\limits_{i=1}^{q}q_i=\sum\limits_{i=1}^{t}(m_i-1)d+tw+V \mbox{ and } B=-C.$$
	
	Also, from \eqref{eq:l1.5}, we get
\begin{equation}\label{eq:Chang1}
\prod\limits_{i=1}^{v}(z+\gamma_i)^{v_i}\prod_{i=1}^{t}(z+\beta_i)^{(m_i-1)d+w}-\prod\limits_{i=1}^{q}(z+\delta_i)^{q_i}=\frac{A}{B}\prod\limits_{i=1}^{s}\left(z+\alpha_i\right)^{(n_i+1)d-w}G_R(z).
\end{equation}
 Put $r=1/z$ in \eqref{eq:Chang1}, we get
\begin{equation}\label{eq:A}
\prod\limits_{i=1}^{v}(1+\gamma_ir)^{v_i}\prod_{i=1}^{t}(1+\beta_ir)^{(m_i-1)d+w}-\prod\limits_{i=1}^{q}(1+\delta_ir)^{q_i}=\frac{A}{B}z^T\prod\limits_{i=1}^{s}\left(1+\alpha_ir\right)^{(n_i+1)d-w}G_R(r), 
	\end{equation}
where $T=\sum\limits_{i=1}^{s}\left((n_j+1)d-w\right)+\mbox{deg}(G_R)-\left[\sum\limits_{i=1}^{t}\left((m_i-1)d+w\right)+V\right].$\\
Put $b(r):
=(A/B)r^{(s+t-1)(w-d)}G_R(1/r).$ Then from $(\ref{eq:A}),$ we have 
\begin{equation*}\frac{\prod\limits_{i=1}^{v}(1+\gamma_ir)^{v_i}\prod_{i=1}^{t}(1+\beta_ir)^{(m_i-1)d+w}}{\prod\limits_{i=1}^{q}(1+\delta_ir)^{q_i}}=r^{(M-N)d+(w-d)+V}\cdot\frac{b(r)\prod\limits_{i=1}^{s}(1+\alpha_ir)^{(n_i+1)d-w}}{\prod\limits_{i=1}^{q}(1+\delta_ir)^{q_i}}+1,
\end{equation*} where $M=\sum\limits_{i=1}^{t}m_i$ and $N=\sum\limits_{i=1}^{s}n_i.$  
Thus we have 
\begin{equation}\label{eq:Chang2}
\frac{\prod\limits_{i=1}^{v}(1+\gamma_ir)^{v_i}\prod_{i=1}^{t}(1+\beta_ir)^{(m_i-1)d+w}}{\prod\limits_{i=1}^{q}(1+\delta_ir)^{q_i}}=1+O\left(r^{(M-N)d+(w-d)+V}\right) \mbox{ as } r\rightarrow 0.
\end{equation}
	 Taking logarithmic derivatives of both sides of \eqref{eq:Chang2}, we obtain
	\begin{equation}\label{eq:Chang3}
		\sum\limits_{i=1}^{v}\frac{v_i\gamma_i}{1+\gamma_ir}+ \sum\limits_{i=1}^{t}\left[(m_i-1)d+w\right]\frac{\beta_i}{1+\beta_ir}-\sum\limits_{i=1}^{q}\frac{q_i\delta_i}{1+\delta_ir}=O\left(r^{(M-N-1)d+w+V-1}\right) \mbox{ as } r\rightarrow 0.
	\end{equation}
Let $E_1=\left\{\gamma_1,\gamma_2,\ldots,\gamma_v\right\}\cap\left\{\beta_1,\beta_2,\ldots,\beta_t\right\}$ and $E_2=\left\{\gamma_1,~\gamma_2,\ldots,\gamma_v\right\}\cap\left\{\delta_1,\delta_2,\ldots,\delta_q\right\}.$ Then the following four cases arise:\\
{\bf  Case 1:} $E_1=E_2=\emptyset$.\\
Let $\beta_{t+i}=\gamma_i$ when $1\leq i\leq v$ and $$T_i=\left\{\begin{array}{cc} (m_i-1)d+w & \mbox{if}~ 1\leq i\leq t,\\ v_{i-t} & \mbox{if}~ t+1\leq i\leq t+v.\end{array}\right.$$ Then \eqref{eq:Chang3} can be written as 
	\begin{equation}\label{eq:l1.8}
		\sum\limits_{i=1}^{t+v}\frac{T_i\beta_i}{1+\beta_ir}-\sum\limits_{i=1}^{q}\frac{q_i\delta_i}{1+\delta_ir}=O\left(r^{M-N-1)d+w+V-1}\right)~\mbox{as $r\rightarrow 0.$}
	\end{equation}
	Comparing the coefficients of $r^j,~j=0,1,\ldots, (M-N-1)d+w+V-2$ in \eqref{eq:l1.8}, we find that 
	\begin{equation}\label{eq:l1.9}
		\sum\limits_{i=1}^{t+v}T_i\beta_i^j-\sum\limits_{i=1}^{q}q_i\delta_i^j=0,~\mbox{for each}~ j=1,2,\ldots, (M-N-1)d+w+V-1.
	\end{equation}
	Now, let $\beta_{t+v+i}=\delta_i$ for $1\leq i\leq q.$ Then from \eqref{eq:l1.9}, we deduce that the system of equations
	\begin{equation}\label{eq:l1.10}
		\sum\limits_{i=1}^{t+v+q}\beta_i^jx_i=0,~j=0,1,\ldots, (M-N-1)d+w+V-1,
	\end{equation}
	has a non-zero solution $$\left(x_1,\ldots,x_{t+v},x_{t+v+1},\ldots,x_{t+v+q}\right)=\left(T_1,\ldots,T_{t+v},-q_1,\ldots,-q_q\right).$$
	This happens only if the rank of the coefficient matrix of the system \eqref{eq:l1.10} is strictly less than $t+v+q.$ Hence $(M-N-1)d+w+V< t+v+q.$ Since $V=\sum\limits_{i=1}^{v}v_i\geq v$ and $M-N\geq t,$ it follows that $q\geq w.$ \\
{\bf  Case 2:} $E_1\neq\emptyset$ and $E_2=\emptyset.$\\
We may assume, without loss of generality, that $E_1=\left\{\gamma_1,\gamma_2,\ldots,\gamma_{s_1}\right\}.$ Then $\gamma_i=\beta_i$ for $1\leq i\leq s_1.$ Take $s_3=v-s_1.$\\
{\bf  Subcase 2.1:} When $s_3\geq 1.$\\
Let $\beta_{t+i}=\gamma_{s_1+i}$ for $1\leq i\leq s_3.$ If $s_1<t,$ then let $$T_i=\left\{\begin{array}{ccc} (m_i-1)d+w+v_i & \mbox{if } 1\leq i\leq s_1,\\  (m_i-1)d+w & \mbox{if } s_1+1\leq i\leq t,\\  v_{s_1-t+i} & \mbox{if } t+1\leq i\leq t+s_3.\end{array}\right.$$ If $s_1=t,$ then we take $$T_i=\left\{\begin{array}{cc} (m_i-1)d+w+v_i & \mbox{if } 1\leq i\leq s_1,\\ v_{s_1-t+i} & \mbox{if } t+1\leq i\leq t+s_3.\end{array}\right.$$
{\bf  Subcase 2.2:} When $s_3=0.$\\
If $s_1<t,$ then set $$T_i=\left\{\begin{array}{cc} (m_i-1)d+w+v_i & \mbox{if } 1\leq i\leq s_1,\\ (m_i-1)d+w & \mbox{if } s_1+1\leq i\leq t\end{array}\right.$$ and if $s_1=t,$ then set $T_i=(m_i-1)d+w+v_i,~\mbox{for } 1\leq i\leq s_1=t.$ 
	
	Thus \eqref{eq:Chang3} can be written as: $$\sum\limits_{i=1}^{t+s_3}\frac{
 T_i\beta_i}{1+\beta_ir}-\sum\limits_{i=1}^{q}\frac{m_i\delta_i}{1+\delta_ir}=O\left(r^{(M-N-1)d+w+V-1}\right)~\mbox{ as }r\rightarrow 0,$$ where $0\leq s_3\leq v-1.$ Proceeding in similar fashion as in Case 1, we deduce that $q\geq w.$\\
{\bf  Case 3:} $E_1=\emptyset$ and $E_2\neq\emptyset.$\\
We may assume, without loss of generality, that $E_2=\left\{\gamma_1,\gamma_2,\ldots,\gamma_{s_2}\right\}.$ Then $\gamma_i=\delta_i$ for $1\leq i\leq s_2.$ Take $s_4=v-s_2.$\\
{\bf  Subcase 3.1:} When $s_4\geq 1.$\\
Let $\delta_{q+i}=\gamma_{s_2+i}$ for $1\leq i\leq s_4.$ If $s_2<q,$ then set $$Q_i=\left\{\begin{array}{ccc} q_i-v_i & \mbox{if } 1\leq i\leq s_2,\\  q_i & \mbox{if } s_2+1\leq i\leq q,\\  -v_{s_2-q+i} & \mbox{if } q+1\leq i\leq q+s_4.\end{array}\right.$$ If $s_2=q,$ then set $$Q_i=\left\{\begin{array}{cc} q_i-v_i & \mbox{if } 1\leq i\leq s_2,\\ -v_{s_2-q+i} & \mbox{if } q+1\leq i\leq q+s_4.\end{array}\right.$$
{\bf  Subcase 3.2:} When $s_4=0.$\\
If $s_2<q,$ then set $$Q_i=\left\{\begin{array}{cc}q_i-v_i & \mbox{if } 1\leq i\leq s_2,\\ q_i & \mbox{if } s_2+1\leq i\leq q\end{array}\right.$$ and if $s_2=q,$ then set $Q_i=q_i-v_i,~\mbox{for } 1\leq i\leq s_2=q.$ 
	
	Thus \eqref{eq:Chang3} can be written as: $$\sum\limits_{i=1}^{t}\frac{\left[(m_i-1)d+w\right]\beta_i}{1+\beta_ir}-\sum\limits_{i=1}^{q+s_4}\frac{Q_i\delta_i}{1+\delta_ir}=O\left(r^{(M-N-1)d+w+V-1}\right)~\mbox{ as }r\rightarrow 0,$$ where $0\leq s_4\leq v-1.$ Proceeding in similar way as in Case 1, we deduce that $q\geq w.$\\
{\bf  Case 4.} $E_1\neq\emptyset$ and $E_2\neq\emptyset.$\\
We may assume, without loss of generality, that $E_1=\left\{\gamma_1,\gamma_2,\ldots,\gamma_{s_1}\right\},~E_2=\left\{\delta_1,\delta_2,\ldots,\delta_{s_2}\right\}.$ Then $\gamma_i=\beta_i$ for $1\leq i\leq s_1$ and $\delta_i=\gamma_{s_1+i}$ for $1\leq i\leq s_2.$ Take $s_5=v-s_2-s_1.$\\
{\bf  Subcase 4.1:} When $s_5\geq 1.$\\
Let $\beta_{t+i}=\gamma_{s_1+s_2+i}$ for $1\leq i\leq s_5.$ If $s_1<t,$ then set $$T_i=\left\{\begin{array}{ccc} (m_i-1)d+w+v_i & \mbox{if } 1\leq i\leq s_1,\\  (m_i-1)d+w & \mbox{if } s_1+1\leq i\leq t,\\  v_{s_1+s_2-t+i} & \mbox{if } t+1\leq i\leq t+s_5.\end{array}\right.$$ If $s_1=t,$ then set $$T_i=\left\{\begin{array}{cc}(m_i-1)d+w+v_i & \mbox{if } 1\leq i\leq s_1,\\ v_{s_1+s_2-t+i} & \mbox{if } t+1\leq i\leq t+s_5.\end{array}\right.$$ If $s_2<q,$ then set $$Q_i=\left\{\begin{array}{cc}q_i-v_{s_1+i} & \mbox{if } 1\leq i\leq s_2,\\ q_i & \mbox{if } s_2+1\leq i\leq q\end{array}\right.$$ and if $s_2=q,$ then set $Q_i=q_i-v_{s_1+i},~\mbox{for } 1\leq i\leq s_2.$\\
{\bf  Subcase 4.2:} When $s_5=0.$\\
If $s_1<t,$ then set $$T_i=\left\{\begin{array}{cc}(m_i-1)d+w+v_i & \mbox{if } 1\leq i\leq s_1,\\ (m_i-1)d+w & \mbox{if } s_1+1\leq i\leq t.\end{array}\right.$$ If $s_1=t,$ then set $T_i=(m_i-1)d+w+v_i~\mbox{for } 1\leq i\leq s_1.$ 
	
Also, if $s_2<q,$ then set $$Q_i=\left\{\begin{array}{cc}q_i-v_{s_1+i} & \mbox{if } 1\leq i\leq s_2,\\ q_i & \mbox{if } s_2+1\leq i\leq q\end{array}\right.$$ and if $s_2=q,$ then set $Q_i=q_i-v_{s_1+i},~\mbox{for } 1\leq i\leq s_2.$
	
Thus in both sub cases, \eqref{eq:Chang3} can be written as $$\sum\limits_{i=1}^{t+s_5}\frac{T_i\beta_i}{1+\beta_ir}-\sum\limits_{i=1}^{q}\frac{Q_i\delta_i}{1+\delta_ir}=O\left(r^{(M-N-1)d+w+V-1}\right)~\mbox{ as }r\rightarrow 0,$$ where $0\leq s_5\leq v-2.$ Using the same arguments as in Case 1, we deduce that $q\geq w.$ This completes the proof.
\end{proof}

\begin{lemma}\label{lem:nikhil}
	Let $f$ be a non-constant non-vanishing rational function, all of whose poles are multiple. Let $Q[f]$ be a differential polynomial of $f$ as defined in \eqref{eqn:3} having weight $w.$ Then, for any $a\in\mC,$ $Q[f](z)-1/(z+a)$ has at least $w$ distinct zeros in $\mathbb{C}.$  
\end{lemma}	

\begin{proof}
The proof can be derived from the proof of Lemma \ref{lem:2} with simple modifications and hence omitted.
\end{proof}

\begin{lemma}\label{lem:3}
Let $\left\{f_j\right\}\subset\mathcal{M}(D)$ be a sequence of non-vanishing functions, all of whose poles are multiple and let $\left\{h_j\right\}\subset\mathcal{H}(D)$ be such that $h_j\rightarrow h$ locally uniformly in $D,$ where $h\in\cH(D)$ and $h(z)\neq 0$ in $D.$ If, for every j, $Q[f_j]-h_j$ has at most $w-1$ zeros, ignoring multiplicities, in $D,$ then $\left\{f_j\right\}$ is normal in $D.$ 
\end{lemma}

\begin{proof}
Suppose that $\left\{f_j\right\}$ is not normal at $z_0\in D.$ Then by Lemma \ref{lem:zp}, there exist a sequence of points $\left\{z_j\right\}\subset D$ with $z_j\rightarrow z_0,$ a sequence of positive real numbers satisfying $\rho_j\rightarrow 0,$ and a subsequence of $\left\{f_j\right\},$ again denoted by $\left\{f_j\right\}$ such that the sequence $$F_j(\zeta):=\frac{f_j(z_j+\rho_j\zeta)}{\rho_j^{(w-d)/d}}\rightarrow F(\zeta),$$ spherically locally uniformly in $\mathbb{C},$ where $F$ is a non-constant and non-vanishing meromorphic function, all of whose poles are multiple. Clearly, $Q[F_j]\rightarrow Q[F]$ spherically uniformly in every compact subset of $\mC$ disjoint from poles of $F.$ Also, one can easily see that $Q[F_j](\zeta)=Q[f_j](z_j+\rho_j\zeta).$ Thus, $$Q[f_j](z_j+\rho_j\zeta)-h_j(z_j+\rho_j\zeta)=Q[F_j](\zeta)-h_j(z_j+\rho_j\zeta)\rightarrow Q[F](\zeta)-h(z_0)$$ spherically locally uniformly in every compact subset of $\mC$ disjoint from the poles of $F.$ Since $F$ is non-constant, by a result of Grahl \cite[Theorem 7]{grahl}, we find that $Q[F]$ is non-constant. Next, we claim that $Q[F]-h(z_0)$ has at most $w-1$ distinct zeros in $\mC.$ 

Suppose that $Q[F]-h(z_0)$ has $w$ distinct zeros, say $\zeta_i,~i=1,~2,~\ldots,~w.$ Then by Hurwitz's theorem, there exist sequences $\zeta_{j, i},~i=1,~2,~\ldots,~w$ with $\zeta_{j, i}\rightarrow\zeta_i$ such that for sufficiently large $j,$ $Q[f_j](z_j+\rho_j\zeta_{j,i})=h_j(z_j+\rho_j\zeta_{j,i})$ for $i=1,~2,~\ldots,~w.$ However, this is not possible since $Q[f_j]-h_j$ has at most $w-1$ distinct zeros in $D.$ This establishes the claim. Now, from Lemma \ref{lem:nik}, it follows that $F$ is a rational function which contradicts Lemma \ref{lem:2}.
\end{proof}

\section{Proof of Theorem \ref{thm:7}}

In view of Lemma \ref{lem:3}, it suffices to prove that $\mathcal{F}$ is normal at points at which $h$ has poles. Since normality is a local property, we can assume $D$ to be $\mD.$ By making standard normalization, it can be assumed that $$h(z)=\frac{1}{z}+a_0+a_1z+\cdots=\frac{b(z)}{z},$$ where $b(0)=1$ and $b(z)\neq 0,~\infty$ in $\mD\setminus\left\{0\right\}.$ Next, we only need to show that $\cF$ is normal at $z=0$. Suppose, on the contrary, that $\cF$ is not normal at $z=0.$ Then by Lemma \ref{lem:zp}, there exists a subsequence  $\left\{f_j\right\}\subset\cF,$ a sequence of points $\left\{z_j\right\}\subset\mD$ with $z_j\rightarrow 0$ and a sequence of positive real numbers $\rho_j$ satisfying $\rho_j\rightarrow 0$ such that the sequence $$g_j(\zeta):=\frac{f_j(z_j+\rho_j\zeta)}{\rho_j^{(w-d-1)/d}}\rightarrow g(\zeta)$$ spherically locally uniformly in $\mathbb{C},$ where $g\in\mathcal{M}(\mC)$ is a non-constant function, all of whose poles are multiple. Also, since each $f_j$ is non-vanishing, it follows that $g$ is non-vanishing.
We now distinguish two cases.\\
{\bf Case 1:} Suppose that there exists a subsequence of $z_j/\rho_j,$ again denoted by $z_j/\rho_j,$ such that $z_j/\rho_j\rightarrow\infty$ as $j\rightarrow\infty.$

Define $$F_j(\zeta):=z_j^{-(w-d-1)/d}f_j(z_j+z_j\zeta).$$ Then a simple computation shows that 
$$Q[F_j](\zeta)=z_jQ[f_j](z_j+z_j\zeta).$$
Also, $F_j$ is non-vanishing, all poles of $F_j$ are multiple, and $b(z_j+z_j\zeta)/(1+\zeta)\rightarrow 1/(1+\zeta)\neq 0$ in $\mD.$ Therefore, 
\begin{align*}
Q[F_j](\zeta)-\frac{b(z_j+z_j\zeta)}{1+\zeta} =z_jQ[f_j](z_j+z_j\zeta)-\frac{b(z_j+z_j\zeta)}{1+\zeta}&=z_j\left[Q[f_j](z_j(1+\zeta))-\frac{b(z_j(1+\zeta))}{z_j(1+\zeta)}\right]\\
&=z_j\left(Q[f_j](z_j(1+\zeta))-h(z_j(1+\zeta))\right)
\end{align*}

Since $Q[f_j]-h$ has at most $w-1$ zeros in $\mD,$ by Lemma \ref{lem:3}, it follows that $\left\{F_j\right\}$ is normal in $\mD$ and so there exists a subsequence of $\left\{F_j\right\},$ again denoted by $\left\{F_j\right\},$ such that $F_j\rightarrow F$ spherically locally uniformly in $\mD,$ where $F\in\cM(\mD)$ or $F\equiv\infty.$  If $F(0)=\infty,$ then 
\begin{align*}
 g_j(\zeta)&=\frac{f_j(z_j+\rho_j\zeta)}{\rho_j^{(w-d-1)/d}}=\left(\frac{z_j}{\rho_j}\right)^{(w-d-1)/d}z_j^{-(w-d-1)/d}f_j(z_j+\rho_j\zeta)=\left(\frac{z_j}{\rho_j}\right)^{(w-d-1)/d}F_j\left(\frac{\rho_j}{z_j}\zeta\right) 
\end{align*}
converges spherically locally uniformly to $\infty$ in $\mC$ showing that $F\equiv\infty,$ a contradiction to the fact that $F$ is non-constant.\\ 
If $F(0)\neq\infty,$ then choose $m\in\mN$ such that $m>(w-d-1)/d$ and let $m_0=m-(w-d-1)/d.$ Then $m_0>0.$ Thus, for each $\zeta\in\mC,$ we have 
\begin{align*}
g_j^{(m)}(\zeta)=\rho_{j}^{m_0}f_j^{(m)}(z_j+\rho_j\zeta)=\left(\frac{\rho_j}{z_j}\right)^{m_0}F_j^{(m)}\left(\frac{\rho_j}{z_j}\zeta\right)\rightarrow 0 \mbox{ as } j\rightarrow\infty,
\end{align*}
which shows that $g$ is a polynomial of degree at most $m-1,$ a contradiction to the fact that $g$ is non-constant and non-vanishing meromorphic function.\\
{\bf Case 2:} Suppose that there exists a subsequence of $z_j/\rho_j,$ again denoted by $z_j/\rho_j,$ such that $z_j/\rho_j\rightarrow\alpha$ as $j\rightarrow\infty,$ where $\alpha\in\mathbb{C}.$ Then 
\begin{align*}
Q[g_j](\zeta)-\rho_j \frac{b(z_j+\rho_j\zeta)}{z_j+\rho_j\zeta}=\rho_j\left(Q[f_j](z_j+\rho_j\zeta)-\frac{b(z_j+\rho_j\zeta)}{z_j+\rho_j\zeta}\right)\rightarrow Q[g](\zeta)-\frac{1}{\zeta+\alpha}
\end{align*}
spherically uniformly on compact subsets of $\mC\setminus\left\{-\alpha\right\}$ disjoint from the poles of $g.$ Since $g$ is non-constant, it can be easily seen that $Q[g]$ is also non-constant. Next, we claim that $Q[g](\zeta)-1/(\zeta+\alpha)$ has at most $w-1$ distinct zeros.\\
Suppose, on the contrary, that $Q[g](\zeta)-1/(\zeta+\alpha)$ has $w$ distinct zeros, say $\zeta_i,~i=1,~2,~\ldots,~w.$ Since $Q[g](\zeta)- 1/(\zeta+\alpha)\not\equiv 0,$ by Hurwitz's theorem, there exist sequences $\zeta_{j, i},~i=1,~2,~\ldots,~w$ with $\zeta_{j, i}\rightarrow\zeta_i$ such that, for sufficiently large $j,$ $$Q[f_j](z_j+\rho_j\zeta_{j,i})-\frac{b(z_j+\rho_j\zeta_{j,i})}{z_j+\rho_j\zeta_{j,i}}= Q[f_j](z_j+\rho_j\zeta_{j,i})-h(z_j+\rho_j\zeta_{j,i})=0$$ for $i=1,~2,~\ldots,~w.$ This is not possible since $Q[f_j]-h$ has at most $w-1$ distinct zeros in $\mD.$ This proves the claim. Now, from Lemma \ref{lem:nik}, it follows that $g$ is a rational function which  contradicts Lemma \ref{lem:nikhil}.

\section{Acknowledgment}
The authors express their gratitude towards the anonymous referee for valuable comments and suggestions that improved the quality and presentation of the paper.

\section{Statements and Declarations}
\begin{itemize}
\item[]{\bf Funding:} The authors declare that no funds, grants, or other support were received during the preparation of this manuscript.

\item[]{\bf Conflict of interests:} The authors declare that they have no conflict of interests. The authors have no relevant financial or non-financial interests to disclose.

\item[]{\bf Data availability:} Data sharing is not applicable to this article as no datasets were generated or analysed during the current study.
\end{itemize}

\bibliographystyle{amsplain}

\end{document}